\documentclass[12pt]{amsart}

\usepackage{amsthm,amsfonts,amsbsy,amssymb,amsmath,amscd}
\usepackage[colorlinks=true]{hyperref}

\begin{document}
\title{Effective bound of linear series on arithmetic surfaces}
\author{Xinyi Yuan}
\author{Tong Zhang}

\address{Department of Mathematics, Princeton University, Princeton, NJ 08544, U.~S.~A.}
\email{yxy@math.princeton.edu}

\address{Department of Mathematics, East China Normal University, Shanghai 200241, People's Republic of China}
\email{mathtzhang@gmail.com}
\maketitle

\theoremstyle{plain}
\newtheorem{theorem}{Theorem}[section]
\newtheorem{lemma}[theorem]{Lemma}
\newtheorem{coro}[theorem]{Corollary}
\newtheorem{prop}[theorem]{Proposition}

\theoremstyle{remark}\newtheorem{remark}[theorem]{Remark}

\theoremstyle{plain}\newtheorem*{thmm}{Theorem}
\theoremstyle{plain}\newtheorem*{thma}{Theorem A}
\theoremstyle{plain}\newtheorem*{thmb}{Theorem B}
\theoremstyle{plain}\newtheorem*{thmc}{Theorem C}
\theoremstyle{plain}\newtheorem*{thmd}{Theorem D}
\theoremstyle{plain}\newtheorem*{core}{Corollary E}

\newcommand{\lb}{\overline{\mathcal{L}}}   
\newcommand{\llb}{\overline{\mathcal{L'}}}
\newcommand{\ob}{\overline{\mathcal{O}}}
\newcommand{\eb}{\overline{\mathcal{E}}}
\newcommand{\fb}{\overline{\mathcal{F}}}

\newcommand{\CL}{{\mathcal{L}}}
\newcommand{\CO}{{\mathcal{O}}}
\newcommand{\CE}{{\mathcal{E}}}

\newcommand{\mb}{\overline{M}}            
\newcommand{\okb}{\overline{O}_K}

\newcommand{\omegaxb}{\overline\omega_X}

\newcommand{\Ar}{{\mathrm{Ar}}}
\newcommand{\Fal}{{\mathrm{Fal}}}

\newcommand{\norm}{\|\cdot\|}

\newcommand{\CC}{{\mathbb{C}}}
\newcommand{\QQ}{{\mathbb{Q}}}
\newcommand{\RR}{{\mathbb{R}}}
\newcommand{\ZZ}{{\mathbb{Z}}}

\newcommand{\hhat}{\widehat h^0}                
\newcommand{\Hhat}{\widehat H^0}
\newcommand{\hsefhat}{\widehat h^0_{\rm sef}}
\newcommand{\Hsefhat}{\widehat H^0_{\rm sef}}

\newcommand{\HPic}{\widehat {\rm Pic}}
\newcommand{\Spec}{\mathrm{Spec}}

\newcommand{\vol}{\mathrm{vol}}             
\newcommand{\volhat}{\widehat{\mathrm{vol}}}             

\newcommand{\chern}{\widehat c_1}           

\tableofcontents

\section{Introduction}

The results of this paper lie in the arithmetic intersection theory of Arakelov,
Faltings and Gillet--Soul\'e.

We prove effective upper bounds on the number of effective
sections of a hermitian line bundle over an arithmetic surface.
The first two results are respectively for general arithmetic divisors
and for nef arithmetic diviors. They can be viewed as effective versions
of the arithmetic Hilbert--Samuel formula.

The third result improves the upper bound substantially for special nef line
bundles, which particularly includes the Arakelov canonical bundle. As a
consequence, we obtain
effective lower bounds on the Faltings height and on the self-intersection of
the canonical bundle in terms of the number of singular points on fibers of the
arithmetic surface. It recovers a result of Bost.

Throughout this paper, $K$ denotes a number field,
and $X$ denotes a regular and geometrically connected arithmetic surface of
genus $g$ over $O_K$ .
That is, $X$ is a two-dimensional regular scheme, projective and flat over
$\Spec
(O_K)$, such that $X_{\overline K}$ is connected curve of genus $g$.

\subsection{Effective bound for arbitrary line bundles}

By a \emph{hermitian line bundle} over $X$, we mean a pair $\lb=(\mathcal L, \|
\cdot \|)$, where
$\CL$ is an invertible sheaf over $X$, and $\| \cdot \|$ is a \emph{continuous
metric}
on the line bundle $\mathcal L(\mathbb C)$ over $X(\mathbb C)$, invariant under
the complex conjugation.

For any hermitian line bundle $\lb=(\mathcal L, \| \cdot \|)$ over $X$, denote
$$
\Hhat(\lb)= \{ s \in H^0(X, \mathcal L): \| s \|_{\sup} \le 1\}.
$$
It is the set of \emph{effective sections}. Define
$$
\hhat(\lb) = \log \# \Hhat(\lb).
$$
and
$$
\volhat(\lb)= \limsup_{n\to \infty} \frac{2}{n^2} \hhat(n\lb).
$$
Here we always write tensor product of (hermitian) line bundles additively, so
$n\lb$ means $\lb^{\otimes n}$.

By Chen \cite{Ch}, the ``limsup'' in the right-hand side is actually a limit.
Thus we have the expansion
$$
\hhat(n\lb)=  \frac{1}{2}\ \volhat(\lb)\ n^2+o(n^2), \quad
n\to\infty.
$$
The first main theorem of this paper is the following \emph{effective} version
of the above expansion in one direction.

\begin{thma}

Let $X$ be a regular and geometrically connected arithmetic surface of genus $g$
over $O_K$. Let $\lb$ be a hermitian line bundle on $X$.
Denote $d^\circ=\deg (\CL_{K})$, and denote by
$r'$ the $O_K$-rank of the $O_K$-submodule of $H^0(\CL)$ generated by
$\Hhat(\lb)$.
Assume $r'\geq 2$.
\begin{itemize}
\item[(1)] If $g>0$, then
$$
\hhat(\lb) \le \frac{1}{2}\ \volhat(\lb)+  4d \log (3d).
$$
Here $d=d^\circ[K:\QQ]$.
 \item[(2)] If $g=0$, then
$$\hhat(\lb)
\leq (\frac12+\frac{1}{r'-1})\ \volhat(\lb)+ 4r \log (3r).
$$
Here $r=(d^\circ+1)[K:\QQ]$.
\end{itemize}

\end{thma}

\subsection{Effective bound for nef line bundles}

Theorem A will be reduced to the case of nef hermitian line bundles.

Recall that a hermitian line bundle $\lb$ over $X$ is \emph{nef} if it satisfies
the following conditions:
\begin{itemize}
\item $\widehat\deg(\lb|_Y) \ge 0$ for any integral subscheme $Y$ of codimension
one in $X$.
\item The metric of $\lb$ is semipositive, i.e., the curvature current of $\lb$
on $X(\CC)$
is semipositive.
\end{itemize}
The conditions imply $\deg (\mathcal L_{K}) \ge 0$.
They also imply that the self-intersection number $\lb^2\geq 0$.
It is a consequence of \cite[Theorem 6.3]{Zh1}. See also \cite[Proposition
2.3]{Mo1}.

The arithmetic nefness is a direct analogue of the nefness in algebraic
geometry. It generalizes the arithmetic ampleness of S. Zhang \cite{Zh1}, and
serves as the limit notion of the arithmetic ampleness.
In particular, a nef hermitian line bundle $\lb$ on $X$ satisfies the
following properties:
\begin{itemize}
\item The degree $\deg (\CL_{K}) \ge 0$, which follows from the definition.
\item The self-intersection number $\lb^2\geq 0$.
It is a consequence of \cite[Theorem 6.3]{Zh1}. See also \cite[Proposition
2.3]{Mo1}.
\item It satisfies the arithmetic Hilbert--Samuel formula
$$
\hhat(n \lb) = \frac{1}{2} n^2\lb^2 + o(n^2), \quad n\to \infty.
$$
Therefore, $\volhat(\lb)=\lb^2$. If $\lb$ is ample in the sense of S. Zhang, the
formula is due to
Gillet--Soul\'e and S. Zhang. See \cite[Corollary 2.7]{Yu1} for more details.
It is extended to nef case by Moriwaki \cite{Mo2} by the continuity
of the volume function.
\end{itemize}

The following result is an \emph{effective} version of the Hilbert--Samuel
formula in one direction.

\begin{thmb}

Let $X$ be a regular and geometrically connected arithmetic surface of genus $g$
over $O_K$.
Let $\lb$ be a nef hermitian line bundle on $X$ with $d^\circ=\deg (\CL_{K})>0$.
\begin{itemize}
\item[(1)] If $g>0$ and $d^\circ>1$, then
$$
\hhat(\lb) \le \frac{1}{2} \lb^2+  4d \log (3d).
$$
Here $d=d^\circ[K:\QQ]$.
 \item[(2)] If $g=0$ and $d^\circ>0$, then
$$\hhat(\lb)
\leq (\frac12+\frac{1}{d^\circ})\lb^2+ 4r \log (3r).
$$
Here $r=(d^\circ+1)[K:\QQ]$.
\end{itemize}

\end{thmb}

The theorem is new even in the case that $\lb$ is ample.
It is not a direct consequence of the arithmetic Riemann--Roch theorem of Gillet
and Soul\'e, due to difficulties on effectively estimating the
analytic torsion and the contribution of $H^1(\CL)$.

Theorem B is a special case of Theorem A under slightly weaker assumptions,
but Theorem B actually implies Theorem A.
To obtain Theorem A, we decompose
$$
\lb=\lb_1+ \eb
$$
with a nef hermitian line bundle $\lb_1$ and and an effective hermitian line
bundle $\eb$,
which induces a bijection $\Hhat(\lb_1)\to \Hhat(\lb)$.
The effectivity of $\eb$ also gives
$\volhat(\lb)\geq \volhat(\lb_1)=\lb_1^2$.
Then the result is obtained by applying Theorem B to $\lb_1$.
See Theorem \ref{degred0}.

The above implication is inspired by the arithmetic Zariski decomposition
of Moriwaki \cite{Mo3}.
Note that Moriwaki's decomposition is a canonical decomposition of $\lb$
of the above form, except that his
$\lb_1$ and $\eb$ usually lie in $\widehat{\mathrm{Pic}}(X)\otimes _{\ZZ}\RR$
instead of
$\widehat{\mathrm{Pic}}(X)$.
One can also extend our result
to $\widehat{\mathrm{Pic}}(X)\otimes _{\ZZ}\RR$
so that Moriwaki's decomposition is applicable.

The following are some consquences and generalities related to the theorems:
\begin{itemize}
\item In the setting of Theorem B, for $\deg(\CL_{K})=1$ and any genus $g\geq
0$, we can bound $\hhat(\lb)$ in terms of $\lb^2$ (with coefficient 1). See
Proposition \ref{deg1}.
\item In both theorems, the assumption that $X$ is regular can be removed by the
resolution of singularity proved by Lipman \cite{Li}.
\item As mentioned above, both theorems can be generalized to $\widehat{\mathrm{Pic}}(X)\otimes
_{\ZZ}\RR$, or equivalently to the group of arithmetic $\RR$-Cartier divisors of
$C^0$-type in the sense of Moriwaki \cite{Mo3}. Furthermore, the results can
also be extended to the setting of adelic line bundles in the sense of S. Zhang
\cite{Zh3}.
\item The theorems easily induce upper bounds for the Euler characteristic
$$\chi_{\sup}(\lb)=\log
\frac{\mathrm{vol}(B_{\sup}(\lb))}{\mathrm{vol}(H^0(X,\CL)_{\RR}/H^0(X,\CL))}.
$$
Here $B_{\sup}(\lb)$ is the unit ball in $H^0(X,\CL)_{\RR}$ bounded by the
supremum norm $\norm_{\sup}$.
In fact, Minkowski's theorem gives
$$
\chi_{\sup}(\lb) \leq \hhat(\lb)+d\log 2.
$$
The bounds are ``accurate'' if $\lb$ is nef.

\end{itemize}

\subsection{Effective bound for special line bundles}

Theorem B is very accurate when $\deg(\CL_K)$ is large by the arithmetic
Hilbert--Samuel formula.
However, it may be too weak if $\deg(\CL_K)$ is very small.
Here we present a substantial improvement of Theorem B for special line bundles,
and consider the application to the Arakelov canonical bundle.

Recall that a line bundle $L$ on a projective and smooth curve over a field
is \emph{special} if both $h^0(L)>0$ and $h^1(L)>0.$
In particular, the canonical bundle is special if the genus is positive.
The following is the improvement of Theorem B in the special case.
One can easily obtain the improvement of Theorem A along the line.

\begin{thmc}
Let $X$ be a regular and geometrically connected arithmetic surface of genus
$g>1$ over $O_K$.
Let $\lb$ be a nef hermitian line bundle on $X$ with $d^\circ=\deg (\CL_{K})>1$.
Assume that $\CL_K$ is a special line bundle on $X_K$.
Then
$$
\hhat(\lb) \le (\frac{1}{4}+\frac{\varepsilon}{2 d^\circ}) \lb^2+  4d \log (3d).
$$
Here $d=d^\circ[K:\QQ]$. The factor $\varepsilon=2$ if $X_K$ is hyperelliptic,
and
$\varepsilon=1$ if $X_K$ is not hyperelliptic.
\end{thmc}

The most interesting case of Theorem C happens when $\lb$ is the canonical
bundle.
Following \cite{Ar}, let $\overline \omega_X=(\omega_X, \norm_{\Ar})$
be the Arakelov canonical bundle of $X$ over $O_K$.
That is, $\omega_X=\omega_{X/O_K}$ is the relative dualizing sheaf of $X$ over
$O_K$ and $\norm_\Ar$ is the Arakelov metric on $\omega_X$.
By Faltings \cite{Fa}, $\overline \omega_X$ is nef if $X$ is semistable over
$O_K$.

\begin{thmd}
Let $X$ be a semistable regular arithmetic surface of genus $g>1$ over $O_K$.
Then
$$
\hhat(\omegaxb) \le \frac{g+\varepsilon-1}{4(g-1)} \omegaxb^2+  4d \log (3d).
$$
Here $d=(2g-2)[K:\QQ]$.
The factor $\varepsilon=2$ if $X_K$ is hyperelliptic, and
$\varepsilon=1$ if $X_K$ is not hyperelliptic.
\end{thmd}

Next we state a consequence of the theorem.
Recall from Faltings \cite{Fa} that $\chi_{\rm Fal}(\omega_X)$
is defined as the arithmetic degree of
the hermitian $O_K$-module $H^0(X,\omega_X)$ endowed with the natural metric
$$
\|\alpha\|_{\rm nat}^2 = \frac i2 \int_{X(\CC)} \alpha \wedge \overline\alpha,
\quad
\alpha\in H^0(X(\CC),\Omega_{X(\CC)}^1).
$$
It is usually called the Faltings height of $X$.
The arithmetic Noether formula proved by Faltings (cf. \cite{Fa, MB}) gives
$$
\chi_{\rm Fal}(\omega_X) = \frac{1}{12} (\overline\omega_X^2 +\delta_X)-\frac13
g[K:\QQ] \log(2\pi).
$$
Here the delta invariant of $X$ is defined by
$$
\delta_X= \sum_{v} \delta_v,
$$
where the summation is over all places $v$ of $K$.
If $v$ is non-archimedean, $\delta_v$ is just the product of $\log q_v$ with the
number of singular points on the fiber of $X$ above $v$. Here $q_v$ denotes the
cardinality of the residue
field of $v$. If $v$ is archimedean, $\delta_v$ is an invariant of the
corresponding Riemann surface.

To state the consequence, we introduce another archimedean invariant.
Let $M$ be a compact Riemann surface of genus $g\geq 1$.
There are two norms on $H^0(M, \Omega_{M}^1)$.
One is the canonical inner product $\|\cdot\|_{\rm nat}$, and the other one is
the supremum norm
$\|\cdot\|_{\sup}$ of the Arakelov metric $\|\cdot\|_{\Ar}$.
Denote by $B_{\rm nat}(\Omega_{M}^1)$ and $B_{\sup}(\Omega_{M}^1)$ the unit
balls in $H^0(M, \Omega_{M}^1)$
corresponding to $\norm_{\rm nat}$ and $\norm_{\sup}$.
Denote
$$
\gamma_M=\frac12 \log \frac{\vol(B_{\rm
nat}(\Omega_{M}^1))}{\vol(B_{\sup}(\Omega_{M}^1))}.
$$
The volumes are defined by choosing a Haar measure on $H^0(M, \Omega_{M}^1)$,
and the quotient does not depend on the choice of the Haar measure.

It is easy to see that both the invariants $\delta$ and $\gamma$ define
real-valued continuous functions
on the moduli space $M_g(\CC)$ of compact Riemann surfaces of genus $g$.

\begin{core}
Let $X$ be a semistable regular arithmetic surface of genus $g>1$ over $O_K$.
Denote
$$\gamma_{X_{\infty}}=\sum_{\sigma:K\hookrightarrow \CC}\gamma_{X_{\sigma}}.$$
Then
\begin{eqnarray*}
 \left(2+\frac{3\varepsilon}{g-1}\right)\overline\omega_X^2
&\geq& \delta_X - 12 \gamma_{X_{\infty}}
-3\ C(g,K),\\
\left(8+\frac{4\varepsilon}{g-1+\varepsilon}\right)\chi_{\rm Fal}(\omega_X)
&\geq& \delta_X - \frac{4(g-1)}{g-1+\varepsilon} \gamma_{X_{\infty}}- C(g,K).
\end{eqnarray*}
Here $\varepsilon$ and $d$ are as in Theorem D, and
$$C(g,K)= 2g\log |D_K|+18 d\log d+25 d,$$
where $D_K$ denotes the absolute discriminant of $K$.
\end{core}

The inequalities are equivalent up to error terms by Faltings's arithmetic
Noether formula.
We describe briefly how to deduce them from Theorem D.
It is standard to use Minkowski's theorem to transfer the upper bound for
$\hhat(\omegaxb)$
to an upper bound for $\chi_{\sup}(\omegaxb)$. It further gives an upper bound
of $\chi_{\rm Fal}(\omega_X)$ since the difference  $\chi_{\rm
Fal}(\omega_X)-\chi_{\sup}(\omegaxb)$ is essentially given by
$\gamma_{X_{\infty}}$.
Now the inequalities are obtained by the arithmetic Noether formula.

Many results of the above type are known in the literature. The new ingredient
of our result is that
the ``error terms'' $\gamma_{X_\infty}$ and $C(g,K)$ are explicitly defined.

Let us first compare the corollary with a result of Bost \cite{Bo}.
The second inequality is an effective version of \cite[Theorem IV]{Bo}.
In particular, if $X_K$ is non-hyperelliptic, then $\varepsilon=1$ and the
coefficient in the corollary agrees with that in \cite{Bo}. Note that our proofs
are completely different.
Bost obtained his result as a special case of his inequality between the slope
of a hermitian vector bundle and the height of a semi-stable cycle, while our
result is a consequence of the estimation of the corresponding linear series.

The first inequality is in the opposite direction of the conjectural arithmetic
Bogomolov--Miyaoka--Yau inequality
proposed by Parshin \cite{Pa}. Recall that the conjecture asserts
$$
\overline\omega_X^2 \leq A (\delta_X + (2g-2) \log |D_K|)+ \sum_{\sigma}
\xi_{X_\sigma}.
$$
Here $A$ is an absolute constant, and $\xi$ is a continuous real-valued function
on $M_g(\CC)$. Note that both $\delta$ and $\gamma$ are such functions.

It will be interesting to compare the first inequality with the main result of
S. Zhang \cite{Zh3}, which proves a formula expressing $\omega_a^2$ in terms of
the Beilinson--Bloch height $\langle\Delta_\xi, \Delta_\xi\rangle$ of the
Gross--Schoen cycle and some
canonical local invariants of $X_K$.
Note that the difference between $\omegaxb^2-\omega_a^2$ is well understood by
\cite{Zh2}.
Then our result gives a lower bound of $\langle\Delta_\xi, \Delta_\xi\rangle$ by
some local invariants.

It is also worth noting that if $X_K$ is hyperelliptic, then $\langle\Delta_\xi,
\Delta_\xi\rangle=0$.
Thus the comparison gives an inequality between two different sums of local
invariants of $X_K$.

\subsection{Classical Noether inequalities}

Let $X$ be a minimal surface of general type over a field, and $\omega_X$ be the
canonical
bundle. The classical Noether inequality
asserts that
$$
h^0(\omega_X) \le \frac{1}{2} \omega_X^2 + 2.
$$
Combining with the Noether formula
$$\chi(\omega_X)= \frac{1}{12}(\omega_X^2+ c_2(\Omega_X)),$$
it is easy to obtain
$$5\ \omega_X^2 \geq c_2(\Omega_X) -36.$$
Theorem D and Corollary E are arithmetic analogue of these results
with sharper coefficients.

In \cite{Sh}, Shin proves that, on an algebraic surface $X$
with non-negative Kodaira dimension,
$$
h^0(L) \le \frac{1}{2} L^2 + 2
$$
for any nef and big line bundle $L$ on $X$ such that the rational map $X
\dashrightarrow \mathbb P(H^0(X,L))$ is generically finite (cf. \cite[Theorem
2]{Sh}).

Theorem B is an arithmetic analogue of Shin's result,
but the proof in \cite{Sh} is not available here due to the essential
use of the adjunction formula.

\subsection{Idea of proof}

Our proofs of Theorem B and Theorem C are very similar.
Theorem C is sharper than Theorem B by the application of Clifford's theorem,
which gives a very good bound on linear series of special line bundles on
curves.

Now we describe the main idea to prove Theorem B.
Let $\lb$ be a nef line bundle.
Denote
$$\Delta(\lb)=\hhat(\lb)-\frac12 \lb^2.$$
We first find the largest constant $c\geq 0$ such that
$$
\lb(-c) = (\mathcal L, e^c\| \cdot \|)
$$
is still nef on $X$.
It is easy to control $\Delta(\lb)$ by $\Delta(\lb(-c))$.
Then the problem is reduced to $\lb(-c)$.

Denote by $\CE_1$ the line bundle associated to the codimension one part of the
base locus of
the strictly effective sections of $\lb(-c)$. We obtain a decomposition
$$
\lb(-c)= \lb_1+\eb_1.
$$
In Theorem \ref{degred1}, we construct hermitian metrics such that $\lb_1$ is
nef and $\eb_1$ is effective,
and such that strictly effective sections of $\lb_1$ can be transfered to those
on $\lb(-c)$.
Then it is easy to control $\Delta(\lb(-c))$ by $\Delta(\lb_1)$.
Then the problem is reduced to $\lb_1$.

The key property for $\lb(-c)$ is that, it usually has a large base locus, due
to the
lack of effective sections. In particular, $\deg(\CE_{1,K})>0$.
It gives a strict inequality $\deg(\CL_{1,K})<\deg(\CL_{K})$.

Keep the reduction process. We obtain $\lb_2, \lb_3, \cdots.$
The process terminates due to the strict decreasing of the degree.
We eventually end up with $\lb_n$ such that $\lb_n(-c_n)$ has no strictly
effective sections.
It leads to the proof of the theorem.

The successive minima of Gillet and Soul\'e is used to control the error terms
in the reduction process.

The structure of the paper is as follows.
In \S 2, we state some results bounding lattice points on normed modules. They
will be used in the proof of the main theorems.
In \S 3, we explore our major construction of the decomposition
$\lb(-c)=\lb_1+\eb_1$, and reduce Theorem A to Theorem B.
In \S 4, we prove Theorem B.
In \S 5, we prove Theorem C and Corollary E.

\

\noindent\textbf{Acknowledgments.}
The authors are very grateful for many important discussions with Atsushi
Moriwaki, Sheng-Li Tan
and Shou-Wu Zhang. In particular, Shou-Wu Zhang's question on applications of
Theorem B motivated
the authors to conceive Theorem C, and the discussions with Moriwaki provided
the bridge from Theorem B to Theorem A.

They would also like to thank the hospitality of Columbia University. The second
author is supported by the Ph.D. Program Scholarship Fund of ECNU, No. 2010025.

\section{Some results on normed modules}

By a normed $\mathbb Z$-module, we mean a pair $\mb=(M,\|\cdot\|)$ consisting of
a $\mathbb Z$-module $M$ and an $\mathbb R$-norm $\| \cdot \|$ on
$M_{\mathbb R}=M \otimes_{\mathbb Z} \mathbb R$.
We say that $\mb$ is a normed free $\mathbb Z$-module of finite rank, if
$M$ is a free $\mathbb Z$-module of finite rank.
It is the case which we will restrict to.

Let $\mb=(M,\|\cdot\|)$ be a normed free $\mathbb Z$-module of finite rank.
Define
$$
\Hhat(\mb)=\{m \in M: \| m \| \le 1\}, \quad
\Hhat_{\rm sef}(\mb)=\{m \in M: \| m \| <1\},
$$
and
$$
\hhat(\mb)= \log \# \Hhat(M), \quad \hhat_{\rm sef}(\mb)= \log \# \Hhat_{\rm
sef}(M).
$$
The Euler characteristic of $\mb$ is defined by
$$\chi(\mb)=\log
\frac{\mathrm{vol}(B(M))}{\mathrm{vol}(M_{\RR}/M)},$$
where
$B(M) = \{ x \in M_\mathbb R: \| x \| \le 1 \}$
is a convex body in $M_{\RR}$.

\subsection{Change of norms}

Let $\mb=(M,\|\cdot\|)$ be a normed free $\mathbb Z$-module of finite rank.
For any $\alpha \in \mathbb R$, define
$$\mb(\alpha)=(M, e^{-\alpha} \| \cdot \| ).$$
Since $\hhat_{\rm sef}(\mb)$ is finite, it is easy to have
$$
\hhat_{\rm sef}(\mb)= \lim_{\alpha\to 0^{\text{-}}}  \hhat(\mb(\alpha)).
$$
Then many results on $\hhat$ can be transfered to $\hsefhat$.
We first present a simple result on the change of effective sections.

\begin{prop} \label{norm1}

Let $\mb=(M,\|\cdot\|)$ be a normed free module of rank $r$. The following are
true:
\begin{itemize}
	\item[(1)] For any $\alpha\geq 0$, one has
$$
\begin{array}{rllll}
\hhat(\mb(-\alpha)) &\leq& \hhat(\mb) &\leq& \hhat(\mb(-\alpha))+
r\alpha+r\log3, \\
\hsefhat(\mb(-\alpha)) &\leq& \hsefhat(\mb) &\leq& \hsefhat(\mb(-\alpha))+
r\alpha+r\log3.
\end{array}
$$
	\item[(2)] One has
	$$ \hsefhat(\mb) \leq \hhat(\mb) \leq \hsefhat(\mb)+r\log3. $$
\end{itemize}
\end{prop}

\begin{proof}

The first inequality of (1) implies the other two inequalities.
In fact, (2) is obtained by setting $\alpha\to 0$ in the first inequality of
(1).
It is also easy to deduce the second inequality of (1) by the first inequality
of (1).
In fact, replace $\mb$ by $\overline M(-\beta)$ with $\beta>0$ in the first
inequality.
Set $\beta\to 0$. The limit gives the second inequality.

Now we prove the first inequality.
For any $\beta>0$, denote
$$B(\beta)=\{x\in M_\RR:\ \|x\| \leq \beta\}.
$$
It is a symmetric convex body in $M_\RR$.
Then $B(1)$ and $B(e^{-\alpha})$ are exactly the unit balls of the metrics of
$\mb$ and $\mb(-\alpha)$.
Consider the set
$$S=\{x+ B(2^{-1}e^{-\alpha}):\ x\in \Hhat(\mb)\}.$$
All convex bodies in $S$ are contained in the convex body
$B(1+2^{-1}e^{-\alpha})$.
Comparing the volumes, we conclude that there is a point $y\in
B(1+2^{-1}e^{-\alpha})$
covered by at least $N$ convex bodies in $S$, where
\begin{eqnarray*}
N &\geq&
\frac{\#S\cdot \vol(B(2^{-1}e^{-\alpha}))}{\vol(B(1+2^{-1}e^{-\alpha}))} \\
&=&\#S\cdot \frac{(2^{-1}e^{-\alpha})^{r}}{(1+2^{-1}e^{-\alpha})^{r}}
=\#S\cdot \frac{1}{(1+2e^{\alpha})^r}.
\end{eqnarray*}
Then
$$
\log N \geq \log\#S -  r\log(1+2e^{\alpha})
\geq \hhat(\mb) -   r(\alpha+\log3).
$$

Let $x_1,\cdots, x_N$ be the centers of these $N$ convex bodies.
Then
$x_i-y\in B(2^{-1}e^{-\alpha})$, and thus
$x_i-x_j\in B(e^{-\alpha})$.
In particular, we have
$$\{x_i-x_1: i=1,\cdots, N\}\subset \Hhat(\mb(-\alpha)).$$
Therefore,
$$\hhat(\mb(-\alpha))\geq \log N
\geq \hhat(\mb) - r(\alpha+\log3).$$
It proves the result.
\end{proof}

\begin{remark}
There are many bounds for $\hhat(\mb)-\hhat(\mb(-\alpha))$ in the literature.
See \cite{GS1, Mo2} for example.
The bound we give here is more accurate. See also \cite[Lemma 2.9]{Yu2}.
\end{remark}

The following filtration version of the proposition will be used in the proof of
our main theorem.

\begin{prop} \label{secred}

Let $\overline M=(M, \|\cdot\|)$ be a normed free $\mathbb Z$-module of finite
rank.
Let $0=\alpha_0\leq \alpha_1\leq \cdots \leq \alpha_n$ be an increasing
sequence.
For $0\leq i \leq n$, denote by $r_i$ the rank of the submodule of $M$ generated
by $\widehat H^0(\overline M(-\alpha_i))$. Then
\begin{eqnarray*}
\widehat h^0(\overline M) & \leq& \widehat h^0(\overline M(-\alpha_n))+
\sum_{i=1}^{n} r_{i-1} (\alpha_{i}-\alpha_{i-1})  + 4r_0 \log r_0+2r_0 \log 3,\\
\widehat h^0(\overline M) &\geq&
\sum_{i=1}^{n} r_i(\alpha_i- \alpha_{i-1})
- 2r_0 \log r_0-r_0 \log 3 .
\end{eqnarray*}
The same results hold for the pair
$(\widehat h^0_{\rm sef}(\overline M),\widehat h^0_{\rm sef}(\overline
M(-\alpha_n))).$
\end{prop}

\begin{remark}
We will only use the first inequality in this paper.
In a forthcoming paper, Yuan will use both inequalities to improve \cite[Theorem
A]{Yu2}.
\end{remark}

The proposition is a consequence of the successive minima of Gillet and Soul\'e.
One may try to use Proposition \ref{norm1} to prove the first inequality.
Namely, for each $i=1,\cdots,n$, one has
$$
h^0(\overline M(-\alpha_{i-1})) \leq \widehat h^0(\overline M(-\alpha_i))+
r_{i-1} (\alpha_{i}-\alpha_{i-1}) + r_{i-1} \log 3.
$$
Summing over $i$, we obtain
$$
\widehat h^0(\overline M) \leq \widehat h^0(\overline M(-\alpha_n))+
\sum_{i=1}^{n} r_{i-1} (\alpha_{i}-\alpha_{i-1}) + (r_0+\cdots+r_{n-1}) \log 3.
$$
The error term may be bigger than that in Proposition \ref{secred}, if the
sequence $\{r_i\}_i$ has
too many terms and decays too slowly. It would actually be the case in our
application.

\subsection{Successive minima}

Here we prove Proposition \ref{secred}. We first recall the successive minima of
Gillet and Soul\'e.

Let $\mb=(M,\|\cdot\|)$ be a normed free $\mathbb Z$-module of finite rank $r$.
For $i=1, \cdots, r$,  {\em the $i$-th logarithmic minimum of $\mb$} is defined
to be
$$
\mu_i(\mb) = \sup \{ \mu \in \mathbb R : {\rm rank} \langle\Hhat(\overline M( -
\mu))\rangle_{\ZZ}
\ge i \}.
$$
Here $\langle\Hhat(\overline M( - \mu))\rangle_{\ZZ}$ denotes the
$\ZZ$-submodule of $M$ generated by
$\Hhat(\overline M(-\mu))$.

The following classical result gives a way to estimate $\hhat(\mb)$ and
$\chi(\mb)$ in terms of the minima of $\mb$.

\begin{theorem}[successive minima]\label{gs}
Let $\mb=(M,\|\cdot\|)$ be a normed free $\mathbb Z$-module of finite rank $r$.
Then
$$
r \log 2-\log(r!)\leq \chi(\mb)- \sum_{i=1}^r \mu_i(\mb) \le r \log 2,
$$
and
$$
\left|\hhat(\mb)- \sum_{i=1}^r \max\{\mu_i(\mb), 0\}\right| \le r \log 3 + 2
r \log r.
$$
The second result still holds if replacing $\hhat(\mb)$ by $\hhat_{\rm
sef}(\mb)$.
\end{theorem}

\begin{proof}

The first result is a restatement of Minkowski's second theorem on successive
minima.

The second result for $\hhat(\mb)$ is essentially due to Gillet--Soul\'e
\cite{GS1}, where the error term is not explicit. It implies the same result for
$\hhat_{\rm sef}(\mb)$. In fact, apply it to $\mb(-\alpha)$ for $\alpha>0$, we
have
$$
\left|\hhat(\mb(-\alpha))- \sum_{i=1}^r \max\{\mu_i(\mb)-\alpha, 0\}\right| \le
r \log 3 + 2
r \log r.
$$
Set $\alpha\to 0$. Note that $\hhat(\mb(-\alpha))$ converges to $\hhat_{\rm
sef}(\mb)$.
It gives the bound for $\hhat_{\rm sef}(\mb)$.

Now we check the explicit error terms.
We will use some effective error terms collected by Moriwaki \cite{Mo2}.
We will use similar notations.

Without loss of generality, assume $M = \mathbb Z^r$.
Define by $M_0$ the submodule of $M$ generated by generated by
$\Hhat(\mb)$, and denote $r_0= \mathrm{rank}(M_0)$.
Denote
$$
B = \{ x \in M_\mathbb R: \| x \| \le 1 \},
$$
which is a convex centrally symmetric bounded absorbing set in $\mathbb R^r$.
Let $B_0=B \cap (M_0 \otimes_\mathbb Z \mathbb R)$ and
let $B_0^*$ be the polar body of $B_0$. That is,
$$
B_0^*=\{x \in M_0 \otimes_\mathbb Z \mathbb R : |\langle x, y \rangle| \le
1  \, {\rm for \, all} \, y \in B_0 \}.
$$
Since $M_0$ is generated by $M_0\cap B_0$, we have $\# (M_0 \cap B_0^*)=1$.

As in \cite{Mo2}, we
have
$$
6^{-r_0} \le \frac{\# \Hhat(\mb)}{\vol(B_0)} \le
\frac{6^{r_0}(r_0!)^2}{4^{r_0}}.
$$
Apply Minkowski's second theorem on successive minima to $B_0$, we obtain
$$
\frac{2^{r_0}}{r_0!} \prod_{i=1}^{r_0} \frac{1}{\lambda_i(B_0)} \le \vol(B_0)
\le
2^{r_0} \prod_{i=1}^{r_0} \frac{1}{\lambda_i(B_0)},
$$
where $\lambda_i(B_0)$ is the $i$-th successive minimum of $B_0$.
Note that we used a different normalization of the minima here, but the relation
is simply $\log \lambda_i(B_0)= \mu_i (\mb)$ for $i = 1, \cdots r_0$.
Thus we can get
$$
\frac{1}{3^{r_0}r_0!}\prod_{i=1}^{r_0} \frac{1}{\lambda_i(B_0)} \le \Hhat(\mb)
\le 3^{r_0} (r_0!)^2 \prod_{i=1}^{r_0} \frac{1}{\lambda_i(B_0)}
$$
Therefore we finally get
\begin{eqnarray*}
\left| \sum_{i=1}^n \max\{\mu_i(\mb), 0\}-\hhat(\mb) \right| & \le & r_0
\log 3 + 2 r_0 \log r_0.
\end{eqnarray*}
It proves the second result.

\end{proof}

\

\begin{proof}[Proof of Proposition \ref{secred}]
By the same limit trick as above, the results for $h^0$ implies that for
$h^0_{\rm sef}$.

We first prove the first inequality.
By definition,
$$
r_0 \geq r_1 \geq \cdots \geq r_n.
$$
If $r_{i-1}=r_{i}$ for some $i$, the inequality does not depend on $\mb_i$.
We can remove $\mb_i$ from the data.
Thus we can assume that
$$
r_0 > r_1 > \cdots > r_n.
$$

For $j=1,\cdots, r_0$, denote by $\mu_j$ the $j$-th logarithmic successive
minimum of $\mb$.
By the definition, it is easy to have
$$
\alpha_{i-1}\leq \mu_{r_{i-1}} \leq \mu_{1+r_i} < \alpha_{i}, \quad i=1,\cdots,
n.
$$
Then we can bound the sequence $\{\mu_j\}_j$ by the sequence $\{\alpha\}_i$.

By Theorem \ref{gs},
$$
\widehat h^0(\overline M) \le \sum_{j=1}^{r_0} \max\{\mu_j,0\} + r_0 \log 3
+ 2r_0 \log r_0.
$$
Replace $\mu_j$ by $\alpha_{i}$ for any $r_i+1 \leq j\leq r_{i-1}$
in the bound.
It gives
$$
\widehat h^0(\overline M) \leq
 \sum_{i=1}^{n} (r_{i-1}-r_i) \alpha_i
 + \sum_{j=1}^{r_n} \max\{\mu_j,0\} +r_0 \log 3 + 2r_0 \log r_0.
$$

Applying Theorem \ref{gs} to $\overline M(-\alpha_n)$, we obtain
\begin{eqnarray*}
\widehat h^0(\overline M(-\alpha_n))
& \ge & \sum_{j=1}^{r_n} \max\{\mu_j-\alpha_n,0\} - r_n \log 3 - 2r_n \log
r_n  \\
& \ge & \sum_{j=1}^{r_n} \max\{\mu_j,0\} - r_n \alpha_n - r_0 \log 3 - 2r_0 \log
r_0.
\end{eqnarray*}
It follows that
\begin{eqnarray*}
\widehat h^0(\overline M)
&\leq&  \sum_{i=1}^{n} (r_{i-1}-r_i) \alpha_i+\widehat h^0(\overline
M(-\alpha_n))
+ r_n\alpha_n \\
& & + 2r_0 \log 3 + 4r_0 \log r_0\\
&=& \widehat h^0(\overline M(-\alpha_n))
+\sum_{i=1}^{n} r_{i-1} (\alpha_{i}-\alpha_{i-1}) \\
& & + 2r_0 \log 3 + 4r_0 \log r_0.
\end{eqnarray*}
It proves the first inequality.

Now we prove the second inequality.
Still apply Theorem \ref{gs}.
We have
$$
\widehat h^0(\overline M) \geq \sum_{j=1}^{r_0} \max\{\mu_j,0\} -r_0 \log 3
- 2r_0 \log r_0.
$$
It follows that
$$
\widehat h^0(\overline M) \geq \sum_{j=1}^{r_1} \max\{\mu_j,0\} -r_0 \log 3
- 2r_0 \log r_0.
$$
Replace $\mu_j$ by $\alpha_{i-1}$ for any $r_i+1 \leq j\leq r_{i-1}$.
It gives
\begin{eqnarray*}
\widehat h^0(\overline M)
& \geq & r_n\alpha_n+ \sum_{i=2}^{n} (r_{i-1}-r_i) \alpha_{i-1}
-r_0 \log 3 - 2r_0 \log r_0\\
&= &  \sum_{i=1}^{n} r_i(\alpha_i- \alpha_{i-1})
-r_0 \log 3 - 2r_0 \log r_0.
\end{eqnarray*}
It finishes the proof.

\end{proof}


\section{The key decompositions}

The key idea of the proof the main theorems is to reduce the sections of $\lb$
to sections of nef line bundles of smaller degree. The goal here is to introduce
this process.

\subsection{Notations and preliminary results}

Let $X$ be an arithmetic surface, and $\lb=(\CL, \| \cdot \|)$ be a hermitian
line bundle over $X$.
We introduce the following notations.

\subsubsection*{Effective sections}
Recall that the set of \emph{effective sections} is
$$
\Hhat(X,\lb)= \{ s \in H^0(X, \mathcal L): \| s \|_{\sup} \le 1\}.
$$
Define the set of \emph{strictly effective sections} to be
$$
\Hhat_{\rm sef}(X,\lb)= \{ s \in H^0(X, \mathcal L): \| s \|_{\sup} < 1\}.
$$
Denote
$$
\hhat(X,\lb)= \log\# \Hhat(X,\lb), \quad \hhat_{\rm sef}(X,\lb)= \log\#
\Hhat_{\rm sef}(X,\lb).
$$
We say that $\lb$ is \emph{effective} (resp. \emph{strictly effective}) if
$\hhat(X,\lb)\neq 0$ (resp. $\hhat_{\rm sef}(X,\lb)\neq0$).

We usually omit $X$ in the above notations. For example, $\Hhat(X,\lb)$ is
written as $\Hhat(\lb)$.

Note that $\overline M=(H^0(X, \mathcal L), \| \cdot \|_{\sup})$ is a normed
$\ZZ$-module.
The definitions are compatible in that
$$
\Hhat(\lb),\quad \Hsefhat(\lb), \quad \hhat(\lb), \quad \hsefhat(\lb)
$$
are identical to
$$
\Hhat(\mb),\quad \Hsefhat(\mb), \quad \hhat(\mb), \quad \hsefhat(\mb).
$$
Hence, the results in last section can be applied here.

For example, Proposition \ref{norm1} gives
$$
\hhat_{\rm sef}(\lb)\leq \hhat(\lb) \leq \hhat_{\rm sef}(\lb)
+ h^0(\CL_\QQ)\log 3.
$$

Note that if $X$ is also defined over $\Spec(O_K)$ for some number field $K$.
Then we obtain two projective curves $X_\QQ=X\times_{\ZZ}\QQ$ and
$X_K=X\times_{O_K}K$,
and two line bundles $\CL_\QQ$ and $\CL_K$.
It is easy to have $h^0(\CL_\QQ)=[K:\QQ]h^0(\CL_K)$.

\subsubsection*{Change of metrics}

For any continuous function $f: X(\mathbb C) \rightarrow \mathbb R$, denote
$$
\lb(f)=(\mathcal{L}, e^{-f}\| \cdot \|).
$$
In particular,
$\ob(f)=(\mathcal{O}, e^{-f})$ is the trivial line bundle with the metric
sending the section 1 to $e^{-f}$.
The case $\ob_X=\ob(0)$ is exactly the trivial hermitian line bundle on $X$.

If $c>0$ is a constant, one has
$$
\hhat(\lb(-c))\leq \hhat(\lb) \leq \hhat(\lb(-c))+h^0(\CL_\QQ)(c+\log 3)
$$
$$
\hhat_{\rm sef}(\lb(-c))\leq \hhat(\lb) \leq \hhat_{\rm
sef}(\lb(-c))+h^0(\CL_\QQ)(c+\log 3).
$$
These also follow from Proposition \ref{norm1}.

\subsubsection*{Base loci}

Let $H$ denote $\Hhat(\lb)$ or $\Hhat_{\rm sef}(\lb)$ in the following.
Consider the natural map
$$
H \times \CL^\vee \longrightarrow \CL \times \CL^\vee \longrightarrow
\mathcal{O}_X.
$$
The image of the composition generates an ideal sheaf of $O_X$.
The zero locus of this ideal sheaf, defined as a closed subscheme of $X$, is
called \emph{the base locus of} $H$
in $X$. The union of the irreducible components of codimension one of the base
locus is called \emph{the fixed part of} $H$ in $X$.

\subsubsection*{Absolute minima}

For any irreducible horizontal divisor $D$ of $X$, define the normalized height
function
$$
h_{\lb}(D)=\frac{\widehat\deg(\lb|_ D)}{\deg D_{\mathbb Q}}.
$$
Define the \emph{absolute minimum} $e_{\lb}$ of $\lb$ to be
$$
e_{\lb}=\inf_{D} h_{\lb}(D).
$$
It is easy to verify that
$$
e_{\lb(\alpha)}=e_{\lb}+\alpha, \quad \alpha\in\RR.
$$

By definition, the absolute minimum $e_{\lb}\geq 0$ if $\lb$ is nef.
Then $\lb(-e_{\lb})$ is a nef line bundle whose absolute minimum is zero.
It is a very important fact in our treatment in the following.

We refer to \cite{Zh1} for more results on the minima of $\lb$ for nef
hermitian line bundles.

\subsection{The key decompositions}

The goal of this section is to prove two basic decompositions of hermitian line
bundles.
They are respectively decompositions keeping $\Hhat(\lb)$
and $\Hhat_{\rm sef}(\lb)$. The proofs are the same, but we state them in
separate theorems
since they will be used for different purposes.

\begin{theorem}\label{degred0}

Let $X$ be a regular arithmetic surface,
and $\lb$ be a nef hermitian line bundle with $\hhat(\lb)\neq 0$.
Then there is a decomposition
$$
\lb=\eb+\lb_1
$$
where $\eb$ is an effective hermitian line bundle on $X$, and $\lb_1$ is a nef
hermitian line bundle on
$X$ satisfying the following conditions:
\begin{itemize}
\item There is an effective section $e\in \Hhat(\eb)$ such that
$\mathrm{div}(e)$ is the fixed part of $\Hhat(\lb)$ in $X$.
\item The map $\CL_1\to \CL$ defined by tensoring with $e$ induces a bijection
$$\Hhat(\lb_1)  \stackrel{\otimes e}{\longrightarrow}\Hhat(\lb).$$
Furthermore, the bijection keeps the supremum norms, i.e.,
$$
\|s\|_{\sup} = \|e\otimes s\|_{\sup}, \quad \forall \ s\in \Hhat(\lb_1).
$$
\end{itemize}

\end{theorem}

\begin{theorem}\label{degred1}

Let $X$ be a regular arithmetic surface,
and $\lb$ be a nef hermitian line bundle with $\hhat_{\rm sef}(\lb)\neq 0$.
Then there is a decomposition
$$
\lb=\eb+\lb_1
$$
where $\eb$ is an effective hermitian line bundle on $X$, and $\lb_1$ is a nef
hermitian line bundle on
$X$ satisfying the following conditions:
\begin{itemize}
\item There is an effective section $e\in \Hhat(\eb)$ such that
$\mathrm{div}(e)$ is the fixed part of $\Hhat_{\rm sef}(\lb)$ in $X$.
\item The map $\CL_1\to \CL$ defined by tensoring with $e$ induces a bijection
$$\Hhat_{\rm sef}(\lb_1)  \stackrel{\otimes e}{\longrightarrow}\Hhat_{\rm
sef}(\lb).$$
Furthermore, the bijection keeps the supremum norms, i.e.,
$$
\|s\|_{\sup} = \|e\otimes s\|_{\sup}, \quad \forall \ s\in \Hhat_{\rm
sef}(\lb_1).
$$
\end{itemize}

\end{theorem}

\

Before proving the theorems, we deduce Theorem A from Theorem B using Theorem
\ref{degred0}.

Let $\lb$ be as in Theorem A. The theorem is trivial if $\hhat(\lb)= 0$.
Assume that $\hhat(\lb)\neq 0$.
As in Theorem \ref{degred0}, decompose
$$
\lb=\lb_1+ \eb.
$$
It particularly gives $\hhat(\lb)= \hhat(\lb_1)$.
For any $n\geq 0$, we have an injection
$$\Hhat(n\lb_1)  \stackrel{ e^{\otimes n}}{\longrightarrow}\Hhat(n\lb).$$
It follows that $\hhat(n\lb)\geq \hhat(n\lb_1)$, and thus
$$\volhat(\lb)\geq \volhat(\lb_1)=\lb_1^2.$$

By $\Hhat(\lb)= \Hhat(\lb_1)$, we have $h^0(\CL_{1,K})\geq r'\geq 2$.
It yields that $\deg(\CL_{1,K})\geq r'\geq 2$ if $g>0$, and $\deg(\CL_{1,K})\geq
r'-1\geq 1$ if $g=0$.
Then we are exactly in the situation to apply Theorem B to $\lb_1$.
It gives exactly Theorem A since $\deg(\CL_{1,K})\leq \deg(\CL_K)=d^\circ$.

\subsection{Construction of the decompositions}

To prove Theorem \ref{degred0} and Theorem \ref{degred1}, we need the following
regularization result in
complex geometry.

\begin{theorem} \label{metric}
Let $L$ be an ample line bundle on a projective complex manifold $M$, and let
$\| \cdot \|$ be any continuous metric on $L$. Then there is a canonical
semipositive continuous metric $\| \cdot \|'$ on $L$ such that $\| \cdot \|' \ge
\| \cdot \|$ pointwise and $\| \cdot \|'_{\rm sup}=\| \cdot \|_{\rm sup}$ as
norms on $H^0(M, m L)$ for any positive integer $m$.
\end{theorem}
\begin{proof}
The metric $\| \cdot \|'$ is constructed as the pluri-subharmonic envelope of
$\| \cdot \|$. See \cite[Theorem 2.3]{Berman} and \cite[Theorem 3.4]{Yu2} for
the case where
$\| \cdot \|$ is smooth. In the continuous case, we find a monotonically
increasing sequence $\{\| \cdot \|_n\}_n$
of smooth metrics which converges uniformly to $\| \cdot \|$.
Then the envelop $\{\| \cdot \|_n'\}_n$ is also monotonically increasing, and
converges uniformly to $\| \cdot \|'$.
It follows that $\| \cdot \|'$ is continuous and semipositive.
\end{proof}

Now we can prove Theorem \ref{degred0} and Theorem \ref{degred1}.
The proof is very similar to the arithmetic Fujita approximation in \cite{Yu2}.
We only consider Theorem \ref{degred1} since Theorem \ref{degred0}
is proved in the same way. We prove Theorem \ref{degred1} by the following
steps.

\subsubsection*{Step 1.}
Denote by $Z$ the fixed part of $\Hhat_{\rm sef}(X, \lb)$.
Set $\CE$ to be the line bundle on $X$ associated to $Z$, and let $e\in
H^0(\CE)$ be
the section defining $Z$. Define a line bundle $\CL_1$ on $X$ by the
decomposition
$$
\CL=\CE+\CL_1.
$$
We need to construct suitable metrics on $\CE$ and $\CL_1$.

\subsubsection*{Step 2.}
There are continuous metrics $\|\cdot\|_{\CE}$ on $\CE$ and
$\|\cdot\|_{\CL_1}$ on $\CL_1$ satisfying:
\begin{itemize}
\item[(a)] $\lb=(\CE,\|\cdot\|_{\CE})+(\CL_1,\|\cdot\|_{\CL_1})$.
\item[(b)] $\|e\|_{\CE, \sup}= 1$.
\item[(c)] The map $\CL_1\to \CL$ defined by tensoring with $e$ induces a
bijection
$$\Hhat_{\rm sef}(\CL_1,\|\cdot\|_{\CL_1})\longrightarrow  \Hhat_{\rm
sef}(\lb)$$
which keeps the supremum norms.
\end{itemize}

The key is that $\Hhat_{\rm sef}(X, \lb)$ is a finite set.
Define the metric  $\|\cdot\|_{\CE}$ on $\CE$ by assigning any $x\in X(\CC)$ to
$$
\|e(x)\|_{\CE}=\max\{ \|s(x)\|/\| s \|_{\sup}:\ s \in \Hhat_{\rm sef}(X, \lb), \
s\neq 0\}.
$$
We claim that it gives a continuous metric on $\CE$.

In fact, it suffices to check that, for some continuous metric
$\|\cdot\|_{\CE,0}$ on $\CE$,
the quotient $\|e(x)\|_{\CE}/\|e(x)\|_{\CE,0}$ extends to a strictly positive
and continuous function of $x\in X(\CC)$.
Denote by $\|\cdot\|_{\CL_1,0}=\|\cdot\|/\|\cdot\|_{\CE,0}$ the quotient metric
on $\CL_1$. Then
$$
\frac{\|e(x)\|_{\CE}}{\|e(x)\|_{\CE,0}}
=\max\left\{ \frac{\|(se^{-1})(x)\|_{\CL_1,0}}{\| s \|_{\sup}}:
\ s \in \Hhat_{\rm sef}(X, \lb), \ s\neq 0\right\}.
$$
It is obviously continuous.
It is strictly positive since the set
$$\{se^{-1}: s \in \Hhat_{\rm sef}(X, \lb)\}$$
is base-point-free on $X(\CC)$,
following from the construction that $\mathrm{div}(e)$ is the codimension-one
part of
the base locus of $\Hhat_{\rm sef}(X, \lb)$ on $X$.

Define the metric $\|\cdot\|_{\CL_1}$ on $\CL_1$ by the decomposition
$$\lb=(\CE,\|\cdot\|_{\CE})+(\CL_1,\|\cdot\|_{\CL_1}).$$
Then (a) and (b) are satisfied.
It is also easy to check (c).
In fact, for any nonzero $s\in \Hhat_{\rm sef}(\lb)$,
we have  $se^{-1} \in H^0(\mathcal{L}')$ by definition.
It suffices to check
$$
\|se^{-1}\|_{\CL_1, \sup} = \| s \|_{\sup}.
$$
By definition of the metrics,
$$
\| s(x) \|\leq \|(se^{-1})(x)\|_{\CL_1} \le \| s \|_{\sup}, \quad x \in
X(\mathbb
C).
$$
Taking supremum, we have
$$
\| s \|_{\sup}\leq \|se^{-1}\|_{\CL_1,\sup} \le \| s \|_{\sup}.
$$
The condition is satisfied.

\subsubsection*{Step 3.}

We can further adjust the metrics in Step 2, such that they keep satisfying
(a), (b), and (c) in Step 2, and such that $\|\cdot\|_{\CL_1}$
is continuous and semipositive.

If $\deg(\CL_{1,\QQ})>0$, applying Theorem \ref{metric} to $\|\cdot\|_{\CL_1}$,
we
obtain a semipositive metric $\|\cdot\|_{\CL_1}'$.
Replace $\|\cdot\|_{\CL_1}$ by $\|\cdot\|_{\CL_1}'$,
and replace $\|\cdot\|_{\CE}$ by the  corresponding quotient metric. It is
easy to see that it keeps all the properties in Step 2.

If $\deg(\CL_{1,\QQ})=0$, then $\CL_{1,\QQ}=\CO_{X_\QQ}$ since $H^0(\CL_1)\neq
0$ by
construction.
Theorem \ref{metric} still holds in this case.
In fact, the metric $\|\cdot\|_{\CL_1}$ is identified with a continuous
function on $X(\CC)$.
Define the metric $\|\cdot\|_{\CL_1}'$ as follows.
On each connected component $M$ of $X(\CC)$, set $\|\cdot\|_{\CL_1}'$ to be the
constant metric given by
the supremum norm $\|1\|_{\CL_1, \sup_M}$ on $M$.
Then the remaining part is as in the case $\deg(\CL_{1,\QQ})>0$.

\subsubsection*{Step 4.}

Finally, we prove that the hermitian line bundle $\lb_1=(\CL_1,
\|\cdot\|_{\CL_1})$ is nef on $X$. By our construction,
We only need to show that $\widehat\deg(\lb_1|_Y)\geq 0$ for any integral
subscheme $Y$.
By definition, the set $\Hhat_{\rm sef}(\lb_1)$ has no fixed part.
For any integral subscheme $Y$ of $X$, we can find a section $s\in\Hhat_{\rm
sef}(\lb_1)$
nonvanishing on $Y$. Then $\widehat\deg(\lb_1|_Y)\geq 0$ by this section.

\section{Proof of Theorem B}

In this section, we use the construction above to prove Theorem B.
We first prove a trivial bound, and then prove the theorem.

\subsection{A trivial Bound}

The following is an easy bound on $\hhat(\lb)$, which serves as the last step of
our reduction.

\begin{prop}\label{easy}

Let $\lb$ be a nef hermitian line bundle on $X$ with $\deg (\CL_K)>0$.
Denote by $r^{\text{-}}$ the $\ZZ$-rank of the $\ZZ$-submodule of $H^0(\CL)$
generated by
$\Hsefhat(\lb)$.
Then we have
$$
\hsefhat(\lb) \le \frac{r^{\text{-}}}{\deg (\CL_\QQ) } \lb^2 +
r^{\text{-}}\log3.
$$
The same result holds for $\hhat(\lb)$ by replacing $r^{\text{-}}$ by the
$\ZZ$-rank of the $\ZZ$-submodule of $H^0(\CL)$ generated by $\Hhat(\lb)$.
\end{prop}

\begin{proof}

Denote
$$\alpha=\frac{1}{r^\text{-}}\hsefhat(\lb)-\log3-\epsilon, \quad \epsilon>0.$$
By Proposition \ref{norm1},
$$
\hsefhat(\lb(-\alpha)) \geq \hsefhat(\lb)- (\alpha+\log3)r^\text{-} = \epsilon
r^\text{-} >0.
$$
It follows that there is a section
$s\in H^0(\CL)$ with
$$-\log\|s\|_{\sup}> \alpha.$$
Then we have
$$
\lb^2= \lb\cdot \mathrm{div}(s)- \int_{X(\CC)} \log\|s\| c_1(\lb) > \alpha\
\deg(\CL_\QQ).
$$
Therefore, we have
$$\frac{1}{r^\text{-}}\hsefhat(\lb)-\log3-\epsilon <
\frac{\lb^2}{\deg(\CL_\QQ)}.$$
Take $\epsilon\to 0$. The inequality follows.

\end{proof}

\begin{remark}
The result can be extended to arithmetic varieties of any dimensions without
extra work.
\end{remark}

\subsection{The reduction process}

Let $\lb$ be a nef line bundle.
We are going to apply Theorem \ref{degred1} to reduce $\lb$ to ``smaller'' nef
line bundles.
The problem is that the fixed part of $\lb$ may be empty, and then
Theorem \ref{degred1} is a trivial decomposition.
The idea is to enlarge the metric of $\lb$ by constant multiples to create base
points.
To keep the nefness, the largest constant multiple we can use gives the case
that the absolute minimum is 0.
The following proposition says that the situation exactly meets our requirement.

\begin{prop}\label{fixedpart}

Let $X$ be a regular arithmetic surface,
and $\lb$ be a nef hermitian line bundle on
$X$ satisfying
$$\hhat_{\rm sef}(\lb)>0, \quad e_{\lb}=0.$$
Then the base locus of $\Hhat_{\rm sef}(\lb)$
contains some horizontal divisor of $X$.

\end{prop}

\begin{proof}

Denote by $S$ the set of horizontal irreducible
divisors $D$ of $X$ such that $h_{\lb}(D)=0$.
The result follows from the properties that $S$ is non-empty and contained in
the base
locus of $\Hhat_{\rm sef}(\lb)$.

First, $S$ is non-empty. Note that the absolute minimum of $\lb$ is 0, so it
suffices to prove that 0 is
not an accumulation point of the range of $h_{\lb}$.
Choose any nonzero section $s \in \Hhat_{\rm sef}(\lb)$.
For any horizontal irreducible divisor $D$ not contained in the support of ${\rm
div}(s)$,
one has
$$
h_{\lb}(D)=\frac{1}{\deg(D_\QQ)} ({\rm div}(s)\cdot D-\log\|s\|(D(\CC))) \geq
-\log\|s\|_{\sup} >0.
$$
It follows that $0$ is not an accumulation point, and there must be an
irreducible component of
$\mathrm{div}(s)$ lying in $X$.

Second, every element of $S$ is contained in the base
locus of $\Hhat_{\rm sef}(\lb)$.
Take any $D\in S$ and $s\in \Hhat_{\rm sef}(\lb)$.
If $s$ does not vanish on $D$, then the above estimate gives
$$
h_{\lb}(D)\geq -\log\|s\|_{\sup} >0.
$$
It is a contradiction.

\end{proof}

Now let us try to prove Theorem B by Theorem \ref{degred1}.
Denote by $c=e_{\lb}$ the absolute minimum.
By definition, $\lb(-c)$ is still nef, and its absolute minimum is 0.
If $\hhat_{\rm sef}(\lb(-c))\neq 0$, applying Theorem \ref{degred1} to
$\lb(-c)$,
we obtain a decomposition
$$
\lb(-c)=\eb+\lb_1
$$
with $\eb$ effective and $\lb_1$ nef, which gives
$$\hhat_{\rm sef}(\lb(-c))=\hhat_{\rm sef}(\lb_1).$$
By Proposition \ref{norm1},
$$
\hhat_{\rm sef}(\lb) \leq \hhat_{\rm sef}(\lb(-c))+ (c+\log 3) h^0(\CL_{\QQ}).
$$

Note that
$$
\lb(-c)^2-\lb_1^2= \eb\cdot (\lb(-c)+\lb_1) \geq 0.
$$
Thus
$$
\lb^2 = \lb(-c)^2+ 2cd \geq \lb_1^2+2cd.
$$
Therefore,
$$
\hhat_{\rm sef}(\lb)-\frac12 \lb^2
\leq \hhat_{\rm sef}(\lb_1)-\frac12 \lb_1^2 + \deg(\CL_{\QQ})\log 3.
$$

By Proposition \ref{fixedpart}, the degree decreases:
$$\deg(\lb_{\QQ})>\deg(\lb_{1,\QQ}).$$
Then we can reduce the theorem for $\lb$ to that for $\lb_1$.
One problem is that, when we keep the reduction process to obtain $\lb_2, \cdots
$, the accumulated error
term
$$
\deg(\CL_{\QQ})\log 3+ \deg(\CL_{1,\QQ})\log 3+\deg(\CL_{2,\QQ})\log 3+\cdots
$$
may grow as
$$
(d+(d-1)+\cdots +1) \log 3=\frac12 d(d+1)\log 3.
$$
It is too big for our consideration.

The key of our solution of the problem is Proposition \ref{secred}.
We put all the sections $\Hhat(\lb_i)$ in one space,
the error term will be decreased to a multiple of $d\log d$ in Lemma
\ref{onestep}. See also the remark after Proposition \ref{secred}.

For convenience of application, we describe the
total construction as a theorem. The proof of Theorem B will be given in
next section.

\begin{theorem}\label{degred}

Let $X$ be a regular arithmetic surface, and let $\lb$ be a nef hermitian line
bundle on
$X$. There is an integer $n\geq 0$, and a sequence of triples
$$
\{(\lb_i, \eb_i, c_i): \ i=0,1, \cdots, n\}
$$
satisfying the following properties:
\begin{itemize}
\item $(\lb_0, \eb_0, c_0)=(\lb, \overline\CO_X, e_{\lb})$.
\item For any $i=0,\cdots, n$, the constant $c_i=e_{\lb_{i}}\geq 0$ is the
absolute minimum of $\lb_{i}$.
\item $\hhat_{\rm sef}(X, \lb_{i}(-c_i))>0$ for any $i=0,\cdots, n-1$.
\item For any $i=0,\cdots,n-1$,
$$
\lb_{i}(-c_{i})=  \lb_{i+1}+ \eb_{i+1}
$$
is a decomposition of $\lb_{i}(-c)$ as in Theorem \ref{degred1}.
\item $\hhat_{\rm sef}(X,\lb_{n}(-e_{\lb_n}))=0.$
\end{itemize}
The following are some properties by the construction:
\begin{itemize}
\item For any $i=0,\cdots, n$, $\lb_i$ is nef and every $\eb_i$ is effective.
\item $\deg(\CL_{0, \mathbb Q})> \deg(\CL_{1, \mathbb Q}) >\cdots > \deg(\CL_{n,
\mathbb Q})$.
\item For any $i=0,\cdots, n-1$, there is a section $e_{i+1}\in
\Hhat(\eb_{i+1})$ inducing a bijection
$$
\Hhat_{\rm sef}(\lb_{i+1})  \longrightarrow  \Hhat_{\rm sef}(\lb_{i}(-c_{i}))
$$
which keeps the supremum norms.
\end{itemize}

\end{theorem}

\begin{proof}
The triple $(\lb_{i+1}, \eb_{i+1}, c_{i+1})$ is obtained by decomposing
$\lb_{i}(-c_{i})$.
The extra part is that Proposition \ref{fixedpart} ensures the degrees on the
generic fiber
decreases strictly.
The process terminates if $\hhat_{\rm sef}(X, \lb_{i}(-c_i))=0.$
It always terminates since $\deg(\CL_{i, \mathbb Q})$ decreases.
\end{proof}

\subsection{Case of positive genus} \label{thmapg}

Here we prove Theorem B in the case $g>0$.
Assume the notations of Theorem \ref{degred}.
We first bound the changes of $\hsefhat(\lb_j)$ and $\lb_j^2 $.

Recall that Theorem \ref{degred} starts with a nef line bundle $\lb_0=\lb$
and constructs the sequence
$$
(\lb_i, \eb_i,c_i), \quad i=0, \cdots, n.
$$
Here $\lb_i$ is nef and $\eb_i$ is effective, and $c_i=e_{\lb_i}\geq 0$.
In particular, $\lb_{i}(-c_{i})$ is still nef.
For any $i=0,\cdots, n-1$, the decomposition
$$
\lb_{i}(-c_i)=  \lb_{i+1}+ \eb_{i+1}
$$
yields a bijection
$$
\Hhat_{\rm sef}(\lb_{i+1})  \longrightarrow  \Hhat_{\rm sef}(\lb_{i}(-c_{i})).
$$
It is given by tensoring some distinguished element $e_i\in \Hhat(\eb_i)$.
It is very important that the bijection keeps the supremum norms.
In the following, we denote
$$
\lb_{i}'=\lb_{i}(-c_i), \quad i=0,\cdots, n.
$$

\begin{prop} \label{onestep}
For any $j=0,\cdots, n$, one has
\begin{eqnarray*}
\lb^2 & \geq &  \lb_j'^2+ 2\sum_{i=0}^{j} d_{i} c_i,\\
\hsefhat(\lb) &\leq& \hsefhat(\lb_j')+
\sum_{i=0}^{j} r_{i} c_i +4r_0 \log r_0+ 2r_0 \log 3.
\end{eqnarray*}
Here we denote $d_{i}=\deg(\mathcal L_{i, \mathbb Q})$ and
$r_i=h^0(\CL_{i,\QQ})$.
\end{prop}

\begin{proof}

Denote $\alpha_{0}=0$ and
$$\alpha_i=c_0+\cdots+c_{i-1}, \quad i=1,\cdots, n.$$
The key is the bijection
$$
\Hhat_{\rm sef}(\lb_{i})\longrightarrow \Hhat_{\rm sef}(\lb(-\alpha_{i})).
$$
It is given by tensoring the section $e_1\otimes \cdots\otimes e_i$.
Denote by $r_i^{\text{-}}$ the rank of the $\mathbb Z$-submodule of $H^0(\CL)$
generated by
$\Hhat(\lb(-\alpha_{i}))$.
Apply Proposition \ref{secred} to $\overline M=(H^0(\mathcal L),
\|\cdot\|_{\sup})$. We obtain
\begin{equation*}
\hsefhat(\lb) \leq \hsefhat(\lb_j(-c_j))+
\sum_{i=0}^{j} r_{i}^{\text{-}} c_i +4r \log r_0^{\text{-}}+ 2r_0^{\text{-}}
\log 3.
\end{equation*}
The result follows since $r_i^{\text{-}}\leq r_i$.

It is also easy to bound the intersection numbers.
In the decomposition
$$\lb_{i}(-c_i)=\lb_{i+1}+ \eb_{i+1},$$
both $\lb_{i}(-c_i)$ and $\lb_{i+1}$ are nef, and $\eb_{i+1}$ is effective.
It follows that
$$
\lb_{i}(-c_i)^2=\lb_{i+1}^2+ (\lb_{i}(-c_i)+\lb_{i+1})\cdot \eb_{i+1} \geq
\lb_{i+1}^2.
$$

By $\lb_{i}(-c_i)^2= \lb_{i}^2 -2d_{i} c_i$, we have
$$
\lb_{i}^2\geq \lb_{i+1}^2+2d_{i} c_i.
$$
Summing over $i=0, \cdots, j-1$, we get the first inequality.
\end{proof}

Now we prove Theorem B. By Proposition \ref{norm1}, it suffices to
prove
$$
\hsefhat(\lb) \le \frac{1}{2} \lb^2+  4d \log d+3d \log 3.
$$
It is classical that $g>0$ implies
$$
r_{i} \leq d_{i}, \quad i=0,\cdots, n-1.
$$
It also holds for $i=n$ if $\deg(\CL_{n,\QQ})\neq 0$.
The the lemma gives, for $j=0,\cdots, n-1$,
\begin{eqnarray} \nonumber
\hsefhat(\lb)-\frac12 \lb^2
\leq \hsefhat(\lb_j')-\frac12 \lb_j'^2 + 4d \log d+ 2d\log 3.
\end{eqnarray}

If $\deg(\CL_{n,\QQ})> 0$, the equality also holds for $j=n$.
Then the theorem is proved since $\hsefhat(\lb_n')=0$ and $\lb_n'^2\geq 0$.

It remains to treat the case $\deg(\CL_{n,\QQ})=0$.
Note that $\CL_{n,\QQ}$ is trivial since $\hsefhat(\lb_n)$
is base-point-free on the generic fiber by construction.
The inequality is not true for $j=n$.
We use the case $j=n-1$ instead.

To bound $\hsefhat(\lb_{n-1}')$, we apply Proposition \ref{easy}.
It gives
$$
\hsefhat(\lb_{n-1}') \le \frac{r_n'}{\deg (\CL_{n-1,\QQ}) } \lb_{n-1}'^2 +
r_n'\log3.
$$
Here $r_n'$ is the $\ZZ$-rank of $\Hsefhat(\lb_{n-1}')=\Hsefhat(\lb_{n})$.
It is easy to have $r_n' \le [K:\QQ]$ since $\CL_{n,\QQ}$ is trivial.

We claim that $\deg (\CL_{n-1, \mathbb{Q}}) \geq 2[K:\QQ]$, or equivalently
$\deg \CL_{n-1, K}\geq 2$.
If $n=1$, it is true by the assumption on $\lb$. Otherwise, by the construction
from
$\lb_{n-2}$,
the base locus of $\Hsefhat(\lb_{n-1})$ is empty or has dimension zero on $X$.
It follows that $\CL_{n-1, K}$ is base-point-free on $X_{K}$.
Its degree is at least two since $g>0$.

By the claim, we have
$$
\hsefhat(\lb_{n-1}') \le \frac12 \lb_{n-1}'^2 + r_n'\log3.
$$
It finishes proving Theorem B for $g>0$.

\subsection{Case of genus zero} \label{genus0}

Here we prove Theorem B in the case $g=0$.
Still apply Proposition \ref{onestep}. We have
\begin{eqnarray*}
\lb^2 & \geq &  \lb_n'^2+ 2\sum_{i=0}^{n} d_{i} c_i,\\
\hsefhat(\lb) &\leq& \hsefhat(\lb_n')+
\sum_{i=0}^{n} r_{i} c_i +4r_0 \log r_0+ 2r_0 \log 3.
\end{eqnarray*}
We still need to compare them.

Denote $\kappa= [K:\QQ].$
Then
$$d_{i}=\deg(\mathcal L_{i, \mathbb Q})=\kappa \deg(\mathcal L_{i, K}), \quad
r_i=h^0(\CL_{i,\QQ})=\kappa h^0(\CL_{i,K}).$$
Note that we do not have $r_i\leq d_i$ any more.
But we have $h^0(\CL_{i,K})=\deg(\mathcal L_{i, K})+1$, and thus
$$r_i= d_i+\kappa, \quad i=0,\cdots, n.$$

Hence, the inequalities yield
\begin{eqnarray*}
&& \hsefhat(\lb)-\frac12 \lb^2 \\
&\leq& \hsefhat(\lb_n')-\frac12 \lb_n'^2 +\kappa\sum_{i=0}^n c_i+ 4r_0 \log r_0+
2r_0\log 3\\
&\leq& \kappa\sum_{i=0}^n c_i+ 4r_0 \log r_0+ 2r_0\log 3.
\end{eqnarray*}
The proof is completed by bound
$$
\sum_{i=0}^n c_i \le \frac{1}{d} \lb^2,
$$
which is implied by the following result.

\begin{lemma} \label{sumci}
In the setting of Theorem \ref{degred}, for any genus $g\geq 0$,
$$
c_0+\sum_{i=0}^n c_i \le \frac{1}{\deg(\CL_\QQ)} \lb^2.
$$
\end{lemma}

\begin{proof}

Denote $\beta=c_1+\cdots+c_n$ and $\fb=\eb_1+\cdots+\eb_n$.
We have the decomposition
$$
\lb_0'(-\beta)=\lb_n'+\fb.
$$
Note that $\lb_0'(-\beta)$ is not nef any more.
But we can still have a weaker bound as follows:
\begin{eqnarray*}
\lb_0'^2
&=& \lb_0'\cdot(\lb_{n}' + \fb+\ob(\beta)) \\
&\geq& \lb_0' \cdot\lb_{n}'+d_0 \beta \\
&=& (\lb_{n}' + \fb+\ob(\beta))\cdot\lb_{n}'+d_0 \beta \\
&\geq& \lb_{n}'^2 + d_{n} \beta+d_0 \beta.
\end{eqnarray*}
Here $d_i=\deg(\CL_{i,\QQ})$ as usual.

Combine with
$$\lb^2= \lb(-c_0)^2+ 2d_0c_0. $$
We have
\begin{equation*}
\lb^2\geq  \lb_n'^2+ d_{n} \beta+d_0 \beta+2d_0c_0
\geq d_0 (2c_0+\beta).
\end{equation*}
The result follows.

\end{proof}

\begin{remark}
The result is in the spirit of the successive minima of S. Zhang \cite{Zh1}.
\end{remark}

The result is stronger and more general than what we need here.
It will be used in the proof of Theorem C in full strength.

\subsection{Extra case of degree one}

\begin{prop} \label{deg1}
Let $X$ be a regular and geometrically connected arithmetic surface of genus $g$
over $O_K$.
Let $\lb$ be a nef hermitian line bundle on $X$ with $\deg(\CL_{K})=1$.

\begin{itemize}
	\item[(1)] If $g>0$, then
$$ \hhat(\lb) \le \lb^2+ [K:\QQ] \log 3.$$
	\item[(2)] If $g=0$, then
$$ \hhat(\lb) \le \lb^2+ 5[K:\QQ] \log 3.$$
\end{itemize}
\end{prop}

\begin{proof}

If $g>0$, the result follows from Proposition \ref{easy}.
If $g=0$, we use the method of \S \ref{thmapg} to get a good bound.
Denote $\kappa=[K:\QQ]$ as usual.
We still have
$$\lb^2=\lb'^2+2\kappa c_0, \quad \hsefhat(\lb)\leq \hsefhat(\lb') +2\kappa c_0+
2\kappa\log 3. $$
It follows that
$$
\hsefhat(\lb) - \lb^2 \leq \hsefhat(\lb')- \lb'^2+ 2\kappa\log 3.
$$
Because the $O_K$-rank $\Hhat(\lb')$ is at most one, Proposition \ref{easy}
gives
$$
\hsefhat(\lb') \leq \lb'^2+ \kappa\log 3.
$$
It follows that
$$
\hsefhat(\lb) \leq \lb^2+ 3\kappa\log 3.
$$
The result follows from Proposition \ref{norm1}.

\end{proof}

\section{Proof of Theorem C}

Our tool to get the strong bound in Theorem C is Clifford's theorem in the
classical setting.
For convenience, we recall it here.

Let $C$ be a projective, smooth and geometrically connected curve over a field
$k$.
Recall that a line bundle $L$ on $C$ is special if
$$h^0(L)>0, \quad h^1(L)>0.$$
The following is Clifford's theorem (cf.\cite[Theorem IV.5.4]{Ha}).

\begin{theorem}[Clifford]
If $L$ is a special line bundle on $C$, then
$$
h^0(L)\leq \frac12 \deg(L)+1.
$$
Furthermore, if $C$ is not hyperelliptic, then the equality is obtained if and
only if
$L\simeq O_C$ or $L\simeq\omega_{C/k}$.
\end{theorem}

\subsection{Hyperelliptic case}

We first prove Theorem C in the hyperelliptic case.
The proof here is very similar to that in \S \ref{genus0}.

By Proposition \ref{onestep},
\begin{eqnarray*}
\lb^2 & \geq &  \lb_n'^2+ 2\sum_{i=0}^{n} d_{i} c_i,\\
\hsefhat(\lb) &\leq& \hsefhat(\lb_n')+
\sum_{i=0}^{n} r_{i} c_i +4r_0 \log r_0+ 2r_0 \log 3.
\end{eqnarray*}
Here $d_{i}=\deg(\mathcal L_{i, \mathbb Q})$ and $r_i=h^0(\CL_{i,\QQ})$.

By construction, each $\CL_{i,K}$ is special.
Clifford's theorem gives
$$
h^0(\CL_{i,K}) \leq \frac12\deg(\mathcal L_{i, K})+1,
$$
and thus
$$
r_i \leq \frac12 d_i+\kappa, \quad \kappa=[K:\QQ].
$$

Hence, the inequalities yield
\begin{eqnarray*}
&&\hsefhat(\lb)-\frac14 \lb^2\\
&\leq& \hsefhat(\lb_n')-\frac14 \lb_n'^2 +\kappa\sum_{i=0}^n c_i+ 4r_0 \log r_0+
2r_0\log 3\\
&\leq& \kappa\sum_{i=0}^n c_i+ 4r_0 \log r_0+ 2r_0\log 3.
\end{eqnarray*}
By Lemma \ref{sumci},
$$
\sum_{i=0}^n c_i \le \frac{1}{d_0}\lb^2.
$$
We obtain
\begin{eqnarray*}
\hsefhat(\lb)\leq \frac14 \lb^2
+ \frac{1}{d^\circ}\lb^2+ 4r_0 \log r_0+ 2r_0\log 3.
\end{eqnarray*}
The result is proved.
The bound and its proof are obviously available for any $g>1$.

\subsection{Non-hyperelliptic case}

For $i=1,\cdots, n-1$, Clifford's theorem gives a stronger bound
$$
h^0(\CL_{i,K}) \leq \frac12\deg(\mathcal L_{i, K})+\frac12,
$$
and thus
$$
r_i \leq \frac12 d_i+\frac12\kappa.
$$
It is also true for $i=0$ or $i=n$ as long as $\CL_{i, K}$ is neither the
canonical bundle nor the trivial bundle.
For $i=0$ or $i=n$, it is always safe to use the bound
$$
r_i \leq \frac12 d_i+\kappa.
$$
The proof of Theorem C is similar,
but more subtle due to the possible failure of the strong bound for $i=0$ and
$i=n$.

We first assume that $\CL_{n, K}$ is non-trivial.
By the strong bounds, Proposition \ref{onestep} gives
\begin{eqnarray*}
&&\hsefhat(\lb)-\frac14 \lb^2\\
&\leq& \hsefhat(\lb_{n}')-\frac14 \lb_{n}'^2 +\frac12\kappa ( c_0+\sum_{i=0}^{n}
c_i)+ 4r_0 \log r_0+ 2r_0\log 3.
\end{eqnarray*}
By Lemma \ref{sumci},
In the setting of Theorem \ref{degred}, for any $g\geq 0$,
$$
c_0+\sum_{i=0}^n c_i \le \frac{1}{d_0}\lb^2.
$$
It follows that
\begin{eqnarray*}
\hsefhat(\lb)\leq
\frac14 \lb^2 +\frac{1}{2d^\circ}\lb^2+ 4r_0 \log r_0+ 2r_0\log 3.
\end{eqnarray*}
It gives
\begin{eqnarray*}
\hhat(\lb)\leq (\frac14 +\frac{1}{2d^\circ})\lb^2+ 4r_0 \log r_0+ 3r_0\log 3.
\end{eqnarray*}

It remains to treat the case that $\CL_{n, K}$ is trivial.
As in the proof of Theorem B in \S \ref{thmapg}, we go back to $n-1$.
Proposition \ref{onestep} gives
\begin{eqnarray*}
&&\hsefhat(\lb)-\frac14 \lb^2\\
&\leq& \hsefhat(\lb_{n-1}')-\frac14 \lb_{n-1}'^2 +\frac12\kappa (
c_0+\sum_{i=0}^{n-1} c_i)+ 4r_0 \log r_0+ 2r_0\log 3.
\end{eqnarray*}

As in \S \ref{thmapg}, $\deg(\CL_{n-1, K})>1$ since $\CL_{n-1, K}$ is
base-point-free by construction, and
$\Hhat(\lb_{n-1}')$ has $\ZZ$-rank at most $\kappa$.
Apply Proposition \ref{easy}. We have
$$
\hsefhat(\lb_{n-1}') \le \frac{1}{2} \lb_{n-1}'^2 + \kappa\log3.
$$
This is actually a special case of Theorem B, but the error term here is better.
Thus
$$
\hsefhat(\lb_{n-1}')-\frac14 \lb_{n-1}'^2 \le \frac{1}{2}\hsefhat(\lb_{n-1}') +
\frac12 \kappa\log3.
$$
By Proposition \ref{norm1},
$$
\hsefhat(\lb_{n-1}')=\hsefhat(\lb_{n}) \leq \hsefhat(\lb_{n}') +\kappa c_n+
\kappa\log3=\kappa c_n+ \kappa\log3.
$$
It follows that
$$
\hsefhat(\lb_{n-1}')-\frac14 \lb_{n-1}'^2 \le \frac{1}{2}\kappa c_n+
\kappa\log3.
$$

Therefore, the bound on $\hsefhat(\lb)$ becomes
\begin{eqnarray*}
&&\hsefhat(\lb)-\frac14 \lb^2\\
&\leq& \frac12\kappa ( c_0+\sum_{i=0}^{n} c_i)+ 4r_0 \log r_0+ 2r_0\log 3+
\kappa\log3.
\end{eqnarray*}
By Lemma \ref{sumci},
$$
c_0+\sum_{i=0}^n c_i \le \frac{1}{d_0}\lb^2.
$$
It follows that
\begin{eqnarray*}
\hsefhat(\lb)\leq
\frac14 \lb^2 +\frac{1}{2d^\circ}\lb^2+ 4r_0 \log r_0+ 2r_0\log 3+\kappa \log3.
\end{eqnarray*}
Thus
\begin{eqnarray*}
\hhat(\lb)\leq (\frac14 +\frac{1}{2d^\circ})\lb^2+ 4r_0 \log r_0+ 4r_0\log 3.
\end{eqnarray*}
It finishes the proof.

\subsection{Application to the canonical bundle}

Theorem D is a special case of Theorem C by Faltings's result that $\omegaxb$ is
nef on $X$.
Corollary E is an easy consequence of Theorem D and Faltings's arithmetic
Noether formula.
Here we briefly track the ``error terms'' in Corollary E.

Recall that $\omegaxb=(\omega_X, \|\cdot\|_{\Ar})$ is endowed with the Arakelov
metric $\|\cdot\|_{\Ar}$.
It induces on $H^0(X,\omega_X)_{\CC}$ the supremum norm
$$
\|\alpha\|_{\sup}=\sup_{z\in M}\|\alpha(z)\|_{\Ar}, \quad \alpha \in H^0(X(\CC),
\Omega_{X(\CC)}^1).
$$
Consider
$$
\chi_{\rm \sup}(\omegaxb)=
\log
\frac{\vol(B_{\sup}(\omega_X))}{\vol(H^0(X,\omega_X)_{\RR}/H^0(X,\omega_X))}.
$$
Here $B_{\sup}(\omega_X)$ is the unit ball in $H^0(X,\omega_X)_{\RR}$ associated
to $\|\cdot\|_{\sup}$.

By Minkowski' theorem,
$$
\hhat(\omegaxb)\geq \chi_{\rm \sup}(\omegaxb)-r\log2.
$$
Here
$$r=g[K:\QQ] \leq (2g-2)[K:\QQ]=d.$$
Thus Theorem D implies
$$
\chi_{\rm \sup}(\omegaxb) \le \frac{g+\varepsilon-1}{4(g-1)} \omegaxb^2+  4d
\log (3d)+r\log2.
$$

Now we compare $\chi_{\rm \sup}(\omegaxb)$ with the Faltings height $\chi_{\rm
Fal}(\omega_X)$.
The latter is the arithmetic degree of the hermitian $O_K$-module
$H^0(X,\omega_X)$
endowed with the natural metric
$$
\|\alpha\|_{\rm nat}^2 = \frac i2 \int_{X(\CC)} \alpha \wedge \overline\alpha,
\quad
\alpha\in H^0(X(\CC),\Omega_{X(\CC)}^1).
$$
By definition, it is easy to obtain
$$
\chi_{\rm Fal}(\omega_X)=
\log \frac{\vol(B_{\rm
nat}(\omega_X))}{\vol(H^0(X,\omega_X)_{\RR}/H^0(X,\omega_X))}
- \chi(O_K^g).
$$
Here $B_{\rm nat}(\omega_X)$ is the unit ball in $H^0(X,\omega_X)_{\RR}$
associated to $\|\cdot\|_{\rm nat}$,
and
$$\chi(O_K^g)=r_1\log V(g)+r_2\log V(2g) - \frac12 g\log |D_K|.$$
Here $D_K$ the absolute discriminant of $K$, $r_1$ (resp. $2r_2$) is
the number of real (resp. complex) embeddings of $K$ in $\CC$,
and $V(r)=\pi^{\frac r2}/\Gamma(\frac r2+1)$ is the volume of the
standard unit ball in the Euclidean space $\RR^r$.

It follows that
$$
\chi_{\rm Fal}(\omega_X)-\chi_{\sup}(\overline\omega_X)
=\log \frac{\vol(B_{\rm nat}(\omega_X))}{\vol(B_{\sup}(\omega_X))}- \chi(O_K^g)
=\gamma_{X_{\infty}}- \chi(O_K^g).
$$
The second equality follows from the definition of
$$\gamma_{X_{\infty}}=\sum_{\sigma:K\hookrightarrow \CC}\gamma_{X_{\sigma}}.$$
Therefore,
$$
\chi_{\rm Fal}(\omega_X) \le \frac{g+\varepsilon-1}{4(g-1)} \omegaxb^2+
\gamma_{X_{\infty}}
- \chi(O_K^g)+ 4d \log (3d)+r\log2.
$$
Stirling's approximation gives
$$\chi(O_K^g)>\frac12 r \log(2\pi) -\frac12 r\log r - \frac12 g\log |D_K|.$$
Thus the inequality implies
$$
\chi_{\rm Fal}(\omega_X) \le \frac{g+\varepsilon-1}{4(g-1)} \omegaxb^2+
\gamma_{X_{\infty}}
+\frac12 g\log |D_K|+\frac92 d\log d+ 4d \log 3.
$$

If $X_K$ is not hyperelliptic, then $\varepsilon=1$ and the coefficient before
$\omegaxb^2$ is exactly the same as the inequality of Bost \cite{Bo} concerning
the height and the slope in this case.
In the following, denote
$$
C'=\frac12 g\log |D_K|+\frac92 d\log d+ 4d \log 3.
$$

Combine with Faltings's arithmetic Noether formula
$$
\chi_{\rm Fal}(\omega_X) = \frac{1}{12} (\overline\omega_X^2 +\delta_X)
-\frac13 r \log(2\pi).
$$
We have
\begin{eqnarray*}
&& \left(2+\frac{3\varepsilon}{g-1}\right)\overline\omega_X^2 \\
&\geq& \delta_X - 12 \gamma_{X_{\infty}}-12 C'- 4r \log(2\pi)\\
&\geq& \delta_X - 12 \gamma_{X_{\infty}}
-6 g\log |D_K|-54 d\log d- 61d.
\end{eqnarray*}
Similarly,
\begin{eqnarray*}
&& \left(8+\frac{4\varepsilon}{g-1+\varepsilon}\right)\chi_{\rm
Fal}(\overline\omega_X) \\
&\geq& \delta_X - \frac{4(g-1)}{g-1+\varepsilon} \gamma_{X_{\infty}}-
\frac{4(g-1)}{g-1+\varepsilon} C'
-4r\log(2\pi)\\
&\geq& \delta_X - \frac{4(g-1)}{g-1+\varepsilon} \gamma_{X_{\infty}}- 4 C'
-4r\log(2\pi)\\
&\geq& \delta_X - \frac{4(g-1)}{g-1+\varepsilon} \gamma_{X_{\infty}}- 2g\log
|D_K|
-18 d\log d-25 d.
\end{eqnarray*}
It completes the inequalities.

\end{document}